\documentclass[a4,11pt, leqno]{article}
\usepackage{amsmath,amsfonts,amsthm,latexsym,amssymb,mathrsfs}
\def\H_0{\mathcal{H}_0(T)}

\def\X{{\mathcal X}}

\def\ind{{\textrm{ind}}}
\def\asc{{\textrm{asc}}}
\def\dsc{{\textrm{dsc}}}
\def\iso{{\textrm{iso }}}
\def\ST{\tau_{AB}}

\def\ind{\textrm{ind}}
\def\X{{\cal X}}
\def\Y{{\cal Y}}
\def\H{{\cal H}}
\def\b{B({\cal X})}
\def\B{B({\cal Y})}

\def\p{{\cal P}}
\def\K{{\cal K}}

\def\iso{\textrm{iso }}

\def\asc{ \textrm{asc}}
\def\dsc{ \textrm{dsc}}
\newtheorem{df}{Definition}[section]
\newtheorem{thm}[df]{Theorem}

\newtheorem{cor}[df]{Corollary}

\newtheorem{rema}[df] {Remark}
\newtheorem{lem}[df] {Lemma}
\def\sfstp{{\hskip-1em}{\bf.}{\hskip1em}}

\def\subject#1{\renewcommand{\thefootnote}{}
\footnote{\it AMS Subject Classification \rm (2010): {#1}}}

\def\keywords#1{\renewcommand{\thefootnote}{}
\footnote{ \it Key words and phrases: {#1}}}

\def\enddemo{\qed \endtrivlist}
\expandafter\let\csname enddemo*\endcsname=\enddemo

\def\qedsymbol{\ifmmode\bgroup\else$\bgroup\aftergroup$\fi
  \vcenter{\hrule\hbox{\vrule
height.6em\kern.6em\vrule}\hrule}\egroup}
\def\qed{\ifmmode\else\unskip\nobreak\fi\quad\qedsymbol}

\begin{document}
\title
{ \bf  Tensor product of  left polaroid operators\/}

\author {\normalsize ENRICO  BOASSO and B. P. DUGGAL } \vskip 1truecm

\date{ }


\maketitle \thispagestyle{empty} 

\subject{Primary 47A80, 47A53. Secondary 47A10.} \keywords{  \rm Banach space,
left polaroid operator, finitely  left polaroid operator, tensor product, \newline
\indent $\hbox{                      \hskip.19truecm }$left-right
multiplication, generalized $a$-Weyl's theorem.}

\vskip 1truecm

\setlength{\baselineskip}{12pt}

\begin{abstract} \noindent A Banach space operator $T\in \b$ is left polaroid if
for each $\lambda\in\iso\sigma_a(T)$ there is an integer
$d(\lambda)$ such that $\asc(T-\lambda)=d(\lambda)<\infty$ and
$(T-\lambda)^{d(\lambda)+1}\X$ is closed; $T$ is finitely left
polaroid if $\asc(T-\lambda)<\infty$, $(T-\lambda)\X$ is closed
and $\dim(T-\lambda)^{-1}(0)<\infty$ at each
$\lambda\in\iso\sigma_a(T)$. The left polaroid property transfers
from $A$ and $B$ to their tensor product $A\otimes B$, hence also
from $A$ and $B^*$ to the left-right multiplication operator
$\tau_{AB}$, for Hilbert space operators; an additional condition
is required for Banach space operators. The finitely left
polaroid property transfers from $A$ and $B$ to their tensor
product $A\otimes B$ if and only if
$0\not\in\iso\sigma_a(A\otimes B)$; a similar result holds for
 $\tau_{AB}$ for finitely left
polaroid $A$ and $B^*$.
\end{abstract}


\section {\sfstp Introduction}\setcounter{df}{0}

 A Banach space operator $T\in B(\X)$ is {\em polar} at a point
 $\lambda$ in its spectrum $\sigma(T)$ if $T-\lambda I$ has both
 finite ascent $\asc(T-\lambda I)$ and descent $\dsc(T-\lambda I)$.
 Apparently, if $T$ is polar at $\lambda\in\sigma(T)$, then
 $\lambda\in\iso\sigma(T)$, the set of isolated points of
 $\sigma(T)$. We say that $T$ is {\em polaroid} if $T$ is polar at
 every $\lambda\in\iso\sigma(T)$. Given Banach spaces $\X$ and
 $\Y$, let $\X\overline{\otimes}\Y$ denote the algebraic completion, endowed
 with a reasonable uniform cross-norm, of the tensor product of
 $\X$ and $\Y$. It is known, \cite[Theorem 3]{DHK}, that the
 polaroid property transfers from $A\in\b$ and $B\in\B$ to their
 tensor product $A\otimes B\in B(\X\overline{\otimes}\Y)$.\par

\indent  $T\in\b$ is {\em left polar} (respectively, {\em right polar}) of order
 $d$ at a point $\lambda$ in its approximate point spectrum $\sigma_a(T)$
 (respectively, surjectivity spectrum $\sigma_s(T)$) if $\asc(T-\lambda
 I)=d<\infty$ and $(T-\lambda I)^{d+1}(\X)$ is closed (respectively,
 $\dsc(T-\lambda I)=d<\infty$ and $(T-\lambda I)^d\X$ is closed).
 It is known that if $T$ is left polar (respectively, right polar) at
 $\lambda$, then $\lambda\in\iso\sigma_a(T)$ (respectively,
 $\lambda\in\iso\sigma_s(T)$). We say that $T$ is {\em left
 polaroid} (respectively, {\em right polaroid}) if $T$ is left polar
 (respectively, right polar) at every $\lambda\in\iso\sigma_a(T)$ (respectively,
 $\lambda\in\iso\sigma_s(T)$). Apparently, $T$ is right polaroid
 if and only if $T^*$ is left polaroid, $T$ is polaroid
 if it is both left and right polaroid and a polaroid operator $T$ is both left and right
 polaroid whenever $\iso\sigma(T)=\iso\sigma_a(T)\cup\iso\sigma_s(T)$. The question that
 we consider here is the following: Does the left polaroid
 property transfer from $A\in\b$ and $B\in\B$ to  $A\otimes B\in
 B(\X\overline{\otimes}\Y)$? The answer to this question is a yes in the case
 in which $\X$ and $\Y$ are Hilbert spaces. In the general case,
 if $A$ and $B$ are  left polar (of order $d(\lambda)$ and $d(\mu)$ at
 points $\lambda\in\iso\sigma_a(A)$ and $\mu\in\iso\sigma_a(B)$),
 and if the closed subspaces $(A-\lambda I)\X + (A-\lambda
 I)^{-d(\lambda)}(0)$ and $(B-\mu I)\Y + (B-\mu
 I)^{-d(\mu)}(0)$ are complemented in $\X$ and $\Y$ respectively for
 every $\lambda\in\iso\sigma_a(A)$ and $\mu\in\iso\sigma_a(B)$,
 then $A\otimes B$ is left polaroid.\par

\indent A stronger version of the left polaroid property says that
 $T\in\b$ is {\em finitely left polaroid} if $T$ is left polar and
 $\alpha(T-\lambda
 I)=\dim(T-\lambda I)^{-1}(0)<\infty$ at every
 $\lambda\in\iso\sigma_a(T)$. The finitely left polaroid property
 transfers from $A\in\b$ and $B\in\B$ to $A\otimes B$ if and only
 if $0\notin\iso\sigma_a(A\otimes B)$. We characterize  $\sigma_a(A\otimes B)$  in terms
of the set of finite left poles and of the Browder essential
approximate point spectrum  of $A$ and of $B$, see section 4.\par

\indent Similar results will be proved for the elementary operator
$\tau_{AB}=L_AR_B$ both in the case of left polaroid operators and
of finitely left polaroid operators $A$ and $B^*$.\par

\section {\sfstp Preliminaries}\setcounter{df}{0}

 Unless otherwise stated, from now on $\X$ (similarly,
$\Y$) shall denote  a complex Banach space and $\b$ (similarly,
$\B$) the algebra of all bounded linear maps defined on and with
values in $\X$ (respectively, $\Y$). Henceforth, {\em we shall
reserve the symbols $T$ and $S$ for general Banach space
operators, and the symbols $A$ and $B$ for operators $A\in\b$ and
$B\in\B$}. Given $T\in \b$, $T^*\in B(\mathcal{X}^*)$ shall
denote the adjoint of $T$, where $\X^*$ is the dual space of
$\X$. Recall that $T\in \b$ is said to be \it bounded below\rm,
if $T^{-1}(0)=\{0\}$ and the range $T\X$ of $T$ is closed. Denote
the \it approximate point spectrum \rm of $T$ by
$\sigma_a(T)=\{\lambda\in \mathbb{C} \colon T-\lambda \hbox{ is
not bounded below} \}$, where $T-\lambda$ stands for $T-\lambda
I$, $I$ the identity map of $\b$. Let $\sigma_s(T)=\{\lambda\in
\mathbb{C} \colon (T-\lambda)\X\neq \X\}$ denote
 the \it surjectivity spectrum \rm of $T$. Clearly,  $\sigma_a(T)\cup  \sigma_s(T)
=\sigma(T)$, the spectrum of $T$. \par

\indent Given $T\in \b$, if $T\X$ is closed and $\alpha (T)=\dim
T^{-1}(0)$ (resp., $\beta (T)=\dim X/T\X$) is finite, then $T$ is
said to be  \it upper semi-Fredholm \rm (resp., \it lower
semi-Fredholm\rm). Moreover, such an operator has a well defined
\it index \rm, i.e., $\hbox{ind}(T)=\alpha(T)-\beta (T)$.
Naturally, from this class of operators the \it upper
semi-Fredholm spectrum \rm can be derived, i. e., the set
$$\sigma_{SF_+}(T)=\{\lambda\in \mathbb C \colon T-\lambda \hbox{ is not upper semi-Fredholm}\}.$$
The lower semi-Fredholm spectrum can be defined in a similar way and it will be
denoted by $\sigma_{SF_-}(T)$.\par

\indent Let $T\in\b$. Recall that   $\asc(T)$ (respectively,
$\dsc(T)$) is the least non-negative integer $n$ such that
$T^{-n}(0)=T^{-(n+1)}(0)$ (respectively, $T^{n}\X=T^{n+1}\X$); if
no such integer exists, then $\asc(T)$ (respectively, $\dsc(T)$)
is infinite. Recall also that $\asc(T-\lambda) <\infty
\Longrightarrow T$ has the single-valued extension property  at
$\lambda$, and that if $T-\lambda$  is upper semi-Fredholm and
has the single-valued extension property (at $0$)  then
$\asc(T-\lambda)<\infty$.   Here, $T$ has the single-valued
extension property at $\lambda$, shortened henceforth to SVEP at
$\lambda$, if, for every open neighbourhood $\cal{U}$ of
$\lambda$, the only analytic function $f: \cal{U}\rightarrow \X$
satisfying $(T-\lambda)f(\lambda)=0$ is the function $f\equiv 0$.
We say that $T$ has SVEP on a subset of the complex plane
$\mathbb{C}$ if it has SVEP at every point of the subset. \par

\indent The \it Weyl essential approximate point spectrum \rm and
the \it Browder essential  approximate point spectrum \rm of
$T\in \b$  are the sets
$$
\sigma_{aw}(T)=\{\lambda\in \sigma_a(T)\colon
T-\lambda\hbox{ is not upper semi-Fredholm or } 0<\hbox{ind} (T-\lambda) \}$$
and
$$\sigma_{ab}(T)=\{\lambda\in \sigma_a(T)\colon \lambda\in \sigma_{aw}(T)
 \hbox{ or } \asc(T-\lambda)=\infty\},$$
respectively. It is clear that
$$\sigma_{SF_+}(T)\subseteq \sigma_{aw}(T)\subseteq \sigma_{ab}(T)\subseteq \sigma_a(T).$$
\indent Concerning the main properties of the aforementioned
spectra, see  \cite{A,R1,R2,R}.\par

\indent  We say that $T\in\b$ is \it semi B-Fredholm\rm, $T\in
\Phi_{SBF}(\X)$, if there exists a non-negative integer $n$ such
that $T^n\X$ is closed and the induced operator
$T_{[n]}=T|_{T^n\X}$ ($T_{[0]}=T$) is semi-Fredholm, upper or
lower, in the usual sense. Observe that $T_{[m]}$ is then
semi-Fredholm for all $m\geq n$: we define the \it index \rm of
$T$ by $\ind(T)=\ind(T_{[n]})$. Let
$$\Phi_{SBF_+^-}(\X)=\{T\in \Phi_{SBF}(\X): T\hspace{2mm}\mbox{is
upper semi B-Fredholm with}\hspace{2mm} \ind(T)\leq 0\};$$ then
the upper semi B-Weyl spectrum of $T$ is the set
$$\sigma_{UBW}(T)=\{\lambda\in\sigma_a(T):
T-\lambda\notin\Phi_{SBF_+^-}(\X)\}.$$
The lower semi B-Weyl spectrum can be defined in a similar way and it will be
denoted by $\sigma_{LBW}(T)$. In addition, $T$ will be said to be
\it $B$-Weyl, \rm if $T$ is both upper and lower semi B-Fredholm
(equivalently $T$ is \it B-Fredholm\rm)
 and
$\ind (T)=0$. The $B$-Weyl spectrum of $T$ is the set
$$
\sigma_{BW}(T)=\{\lambda\in\sigma(T): T-\lambda \hbox{ is not
$B$-Fredholm or }\ind(T-\lambda)\neq 0\}.
$$
Note that $\sigma_{LBW}(T)=\sigma_{UBW}(T^*)$ and
$\sigma_{BW}(T)=\sigma_{UBW}(T)\cup\sigma_{LBW}(T)$.\par

\indent We say that $T$ is \it
quasi-Fredholm  of degree $d$ ($\geq 0$)\rm,  if $\dim(T^n\X \cap
T^{-1}(0))\setminus (T^{n+1}\X \cap T^{-1}(0)) =0$ for all $n\geq
d$, and the subspaces $T^{-d}(0)+ T\X$ and $T^{-1}(0)\cap T^d\X$
are closed. Every semi B-Fredholm operator is quasi-Fredholm
\cite{BS}.\par

\bigskip
\indent Let $\Pi^{\ell}(T)$ denote the set of left poles of $T\in
\b$, i.e., $\Pi^{\ell}(T)= \{\lambda\in \sigma_a(T)\colon \asc
(T-\lambda)=d<\infty \hbox{ and } (T-\lambda)^{d+1}\X \hbox{ is
closed}\}$.
 If
$\lambda\in \Pi^{\ell}(T)$ is a left pole of order $d$, then
$\lambda\in\iso\sigma_a(T)$, $\lambda\notin\sigma_{UBW}(T)$
(\cite[Lemma 3.1]{D1}), and $(T-\lambda)_{[d]}=
(T-\lambda)|_{(T-\lambda)^d\X}$ is bounded below (\cite[Theorem
2.5]{ABC}), where if $M\subseteq \mathbb C$, then acc $M$ stands
for the set of limit points of $M$ and iso $M=M\setminus$ acc
$M$. (Indeed, $\lambda\in\Pi^{\ell}(T)$ if and only if
$\lambda\notin\sigma_{UBW}(T)$ and $T$ has SVEP at $\lambda$.)
Furthermore, if we let
$$H_0(T)=\{x\in\X:\textrm{lim}_{n\longrightarrow\infty}{||T^n
x||^{\frac{1}{n}}=0}\}$$ denote the
 quasi--nilpotent part of $T$, then
 $H_0(T-\lambda)=(T-\lambda)^{-d}(0)$ (\cite[Theorem 2.3]{A1}).

 \indent  It is known, \cite[Lemma 3.5]{D1}, that if
$T^*$ has SVEP at points $\lambda\notin\sigma_{UBW}(T)$, then
$\sigma_{UBW}(T)=\sigma_{BW}(T)$. This implies that if
$\lambda\in\Pi^{\ell}(T)$ and $T^*$ has SVEP at points
$\lambda\notin\sigma_{UBW}(T)$, then $\lambda\in\iso\sigma(T)$
and $T$ is polar at $\lambda$ (\cite[Corollary 3.13]{D1}).
Consequently, if $T^*$ has SVEP at points
$\lambda\in\Pi^{\ell}(T)$, then
$\Pi^{\ell}(T)=\Pi(T)=\{\lambda\in\iso\sigma(T): T$ is polar at
$\lambda\}$.
\markright{\hskip4.2truecm \rm Tensor product of  left polaroid operators }\par
\indent Concerning finitely left polaroid operators,  recall from \cite[Theorem 3.8]{G} that if $T\in \b$,
 $\alpha(T)< \infty$ and $\asc(T)< \infty$, then $T^n\X$
 is closed for some integer $n> 1$ if and only if $T\X$ is closed. Hence $T$ is
 finitely left polaroid if and only if $\alpha(T-\lambda I)< \infty$,
 $\asc(T-\lambda I)< \infty$ and $(T-\lambda I)\X$ is closed for
 every $\lambda\in\iso\sigma_a(T)$. Let $\Pi^{\ell}_0(T)$ denote the set of finite left poles of $T$, i.e.,
$$
\Pi^{\ell}_0(T)=\{\lambda\in\iso\sigma_a(T): T-\lambda
 \hspace{2mm}\mbox{is upper semi-Fredholm and}\hspace{2mm}
 \asc(T-\lambda) <\infty \}.
$$
Then $T\in \b$ is finitely left polaroid if and only if $\iso
\sigma_a(T)=\Pi^{\ell}_0(T)$\par

\indent In the following remark, several properties of finite left poles
will be recalled.\par
\begin{rema}\label{remark1} \rm  Let $T\in \b$. Then $\sigma_a(T)\setminus \sigma_{ab}(T)
=\Pi_0^{\ell}(T)$ (\cite[Corollary 2.2]{R}).
Additionally, if $\lambda \in \hbox{iso}\hskip.1truecm
\sigma_a(T)$, then the following statements are equivalent:
$$
\hbox{(i)} \hskip.1truecm \lambda\in \sigma_a(T)\setminus \sigma_{ab}(T),\hskip1truecm
\hbox{(ii)} \hskip.1truecm \lambda\in \sigma_a(T)\setminus \sigma_{aw}(T).$$

\indent As a result, if we let $I_0^a(T)= \hbox{iso}\hskip.1truecm
\sigma_a(T)\setminus\Pi_0^l(T)$, then (since
$\sigma_{aw}(T)\subseteq  \sigma_{ab}(T)$),
$$
I_0^a(T)\subseteq  \sigma_{aw}(T)\subseteq I_0^a(T)\cup \hbox{acc}\hskip.1truecm \sigma_a(T)=
 \sigma_{ab}(T)
$$ (\cite[Theorem 2.1 and Corollary
2.3]{R}). Therefore, necessary and sufficient for $T$ to be
finitely left polaroid is that one of the following statements
holds:
\begin{align*}
&\hbox{(iii)}\hskip.1truecm \sigma_{ab}(T)=\hbox{acc}\hskip.1truecm \sigma_a(T),\hskip.2truecm
\hbox{(iv)}\hskip.1truecm  \sigma_{aw}(T)=\hbox{acc}\hskip.1truecm \sigma_a(T)\cap  \sigma_{aw}(T),&\\
&\hbox{(v)} \hskip.1truecm\sigma_a(T)=\sigma_{ab}(T)\cup
\hbox{iso}\hskip.1truecm \sigma_a(T),\hskip.2truecm \sigma_{ab}(A)
\cap \hbox{iso}\hskip.1truecm \sigma_a(A)=\emptyset.&\\
\end{align*}

\indent Note that if $\sigma_a(T)= \{ 0\}$, in particular if $T$
is a quasi-nilpotent operator, then $T$ is not finitely left
polaroid. In fact, if $\sigma_a(T)= \{ 0\}$ and $T$ is finitely
left polaroid, then
 $$
\emptyset=\hbox{acc}\hskip.1truecm \sigma_a(T)=\sigma_{ab}(T).$$
\noindent Since $\sigma_{ab}(T)\neq\emptyset$  (\cite[Corollary
2.4]{R} and \cite[Theorem 1]{R1}), this is a contradiction.

\end{rema}
\medskip

\medskip
\indent Let $\X\overline{\otimes}\Y$ denote the completion of the
algebraic tensor product  of $\X$ and $\Y$, $\X\otimes \Y$,
relative to some reasonable cross norm; let
 $A\otimes B\in B( \X\overline{\otimes}\Y)$  denote the tensor
product of $A$ and $B$. Then, \cite[Lemma 5]{DDK},
$$
\sigma_{ab}(A\otimes
B)=\sigma_a(A)\sigma_{ab}(B)\cup\sigma_{ab}(A)\sigma_a(B).
$$

\noindent Again, if   $\tau_{AB}=L_AR_B\in B(B(\Y,\X))$ denotes
the elementary operator
$$\tau_{AB}(X)=L_AR_B(X)=AXB,$$  then $$
\sigma_{ab}(\tau_{AB})=\sigma_a(A)\sigma_{ab}(B^*)\cup\sigma_{ab}(A)\sigma_a(B^*),
$$ \cite[Proposition 4.3 (iv)]{BDJ}.
\medskip

\indent The following lemma studies the sets of the limit and
the isolated points of the operators considered in this
article.\par
 \begin{lem}\label{lemma1} If $A$ and $B$ are
finitely left polaroid, then the following  statements hold.\rm\par
\noindent (i) acc $\sigma_a(A\otimes B)\subseteq \sigma_{ab}(A\otimes B)\subseteq$
acc $\sigma_a(A\otimes B)\cup \{0 \}$;\par
\noindent (ii) $\iso \sigma_a(A\otimes B)\setminus \{0\}\subseteq \Pi_0^l(A)\cdot\Pi_0^l(B)\subseteq
\iso\sigma_a(A\otimes B)\cup \{0 \}.$\par
\indent \it If, instead, $A$ and $B^*$ are finitely  left
polaroid, then the following statements hold.\rm\par
\noindent (iii) acc $\sigma_a(\tau_{AB})\subseteq \sigma_{ab}(\tau_{AB})\subseteq$
acc $\sigma_a(\tau_{AB})\cup \{0 \}$;\par
\noindent (iv) $\iso \sigma_a(\tau_{AB})\setminus \{0\}\subseteq \Pi_0^l(A)\cdot\Pi_0^a(B^*)\subseteq
\iso \sigma_a(\tau_{AB})\cup \{0 \}$.
\end{lem}
\begin{proof} Since $\hbox{iso}\hskip.1truecm
\sigma_a(A)=\Pi_0^l(A)$ and $\hbox{iso}\hskip.1truecm
\sigma_a(B)=\Pi_0^l(B)$ (Remark \ref{remark1}), the proof of $(i)$ and $(ii)$ is
immediate from \cite[Theorem 6]{HK} once one observes that
$\sigma_a(A\otimes B)=\sigma_a(A)\cdot \sigma_a(B)$ (\cite[Theorem
4.4]{I}), $\sigma_{ab}(A)=\hbox{acc}\hskip.1truecm \sigma_a(A)$,
$\sigma_{ab}(B)=\hbox{acc}\hskip.1truecm \sigma_a(B)$ (Remark
\ref{remark1}), and $\sigma_{ab}(A\otimes
B)=\sigma_a(A)\sigma_{ab}(B)\cup\sigma_{ab}(A)\sigma_a(B)$
(\cite[Lemma 5]{DDK}).\par
\markright{ \hskip5truecm \rm ENRICO BOASSO and B. P. DUGGAL}
\indent One argues similarly  to prove $(iii)$ and $(iv)$:
observe that  $\hbox{iso}\hskip.1truecm \sigma_a(A)=\Pi_0^l(A)$
and $\hbox{iso}\hskip.1truecm \sigma_a(B^*)=\Pi_0^l(B^*)$,
$\sigma_a(\tau_{AB})=\sigma_a(A)\cdot \sigma_a(B^*)$
(\cite[Proposition 4.3 (i)]{BDJ}) and
$\sigma_{ab}(\tau_{AB})=\sigma_a(A)\sigma_{ab}(B^*)\cup\sigma_{ab}(A)\sigma_a(B^*)
$  (\cite[Proposition 4.3 (iv)]{BDJ}).
\end{proof}

\begin{rema}\label{remark101} \rm Note that under the conditions of Lemma \ref{lemma1}
$$
\sigma_{ab}(A\otimes B)=\sigma_{ab}(A)\cdot \sigma_{ab}(B)\cup
\sigma_{ab}(A)\cdot \Pi^l_0 (B)\cup \Pi^l_0(A)\cdot\sigma_{ab}(B)
$$  (\cite[Lemma 5]{DDK}). Similarly
$$
\sigma_{ab}(\ST)=\sigma_{ab}(A)\cdot \sigma_{ab}(B^*)\cup
\sigma_{ab}(A)\cdot \Pi^l_0 (B^*)\cup
\Pi^l_0(A)\cdot\sigma_{ab}(B^*)
$$  (\cite[Proposition 4.3 (iv)]{BDJ}).
\end{rema}

\section {\sfstp Left polaroid operators}\setcounter{df}{0}

 We say that a left polar
operator $T\in\b$, of order $d(\lambda)$ at
$\lambda\in\iso\sigma_a(T)$, satisfies property $(\p)$ if the
closed subspace $(T-\lambda)^{-d(\lambda)}(0) + (T-\lambda)\X$ is
complemented in $\X$ for every $\lambda\in\iso\sigma_a(T)$. The
following lemma proves that left polaroid operators satisfying
property $\p$ have a Kato type decomposition.
\par
\markright{\hskip4.2truecm \rm Tensor product of  left polaroid operators }
\begin{lem}\label{lem1} If $T\in\b$ is left polaroid and satisfies
property $(\p)$, then for every $\lambda\in\iso\sigma_a(T)$ there
exist $T$-invariant closed subspaces $E_1$ and $E_2$ such that
$\X=E_1\oplus E_2$,
$H_0(T-\lambda)=(T-\lambda)^{-d(\lambda)}(0)=H_0((T-\lambda)|_{E_1})$
and $(T-\lambda)|_{E_2}$ is bounded below, where $d(\lambda)$ is
the order of the left pole at $\lambda$.
\end{lem}

\begin{demo} The hypotheses imply that $T-\lambda$ is
quasi-Fredholm of order $d(\lambda)$, and the closed subspaces
$(T-\lambda)^{-d(\lambda)}(0) + (T-\lambda)\X$ and
$(T-\lambda)^{-1}(0)\cap (T-\lambda)^{d(\lambda)}\X$ are
complemented in $\X$. Hence, \cite[Theorem 5]{Mu}, there exist
$T$-invariant closed subspaces $E_1$ and $E_2$ such that
$\X=E_1\oplus E_2$, $(T-\lambda)^{d(\lambda)}|_{E_1}=0$ and
$(T-\lambda)|_{E_2}$ is semi-regular. (Recall, \cite[Page 7]{A},
that $T-\lambda$ is semi-regular if $(T-\lambda)\X$ is closed and
$(T-\lambda)^{-n}(0)\subseteq (T-\lambda)^m\X$ for all natural
numbers $m,n$.) Since
$\asc(T-\lambda)=d(\lambda)<\infty\Longleftrightarrow
(T-\lambda)^{d(\lambda)}\X\cap (T-\lambda)^{-n}(0)=\{0\}$ for
every natural number $n$, the semi-regular operator
$(T-\lambda)|_{E_2}$ is injective. Hence $(T-\lambda)|_{E_2}$ is
bounded below. Observe that
\begin{eqnarray*} H_0(T-\lambda) & = &
H_0((T-\lambda)|_{E_1}) \oplus H_0((T-\lambda)|_{E_2})\\ & = & E_1
\oplus 0= E_1.\end{eqnarray*} This, since
$H_0(T-\lambda)=(T-\lambda)^{-d(\lambda)}(0)$ by \cite[Theorem
2.3]{A1}, completes the proof.\end{demo}

\indent Next follows the main result of this section.\par

\begin{thm}\label{thm1}
Let $A$ and $B$ be left polaroid operators. If $A$ and $B$
satisfy property $(\p)$, or if $\X$ and $\Y$ are Hilbert spaces,
then $A\otimes B$ is left polaroid.\end{thm}\begin{demo} We
consider the case in which $\X, \Y$ are Banach spaces and $A, B$
satisfy property $(\p)$; since $A, B$ automatically satisfy
property $(\p)$ in the case in which $\X, \Y$ are Hilbert spaces,
the proof for the Hilbert space case is a consequence of the
Banach space case.\par

\indent Since $\sigma_a(A\otimes B)=\sigma_a(A)\sigma_a(B)$,
$\iso\sigma_a(A\otimes B)= \iso(\sigma_a(A)\sigma_a(B))\subseteq
\iso\sigma_a(A) \iso\sigma_a(B)$; furthermore, this is easily
seen, $\iso\sigma_a(A\otimes B)\setminus\{0\}\subseteq
\iso\sigma_a(A) \iso\sigma_a(B)\subseteq \iso\sigma_a(A\otimes
B)\cup\{0\}$. We consider the cases $(i)\hspace{2mm}
0\neq\lambda\in\iso\sigma_a(A\otimes B)$ and $(ii)\hspace{2mm}
0=\lambda\in\iso\sigma_a(A\otimes B)$ separately.\par

\indent $(i).$ In this case, for every $\lambda\in\iso\sigma_a(A\otimes
B)$ there exist non-zero $\mu\in\iso\sigma_a(A)$ and
$\nu\in\iso\sigma_a(B)$ such that $\mu \nu=\lambda$. The operator
$A$ and $B$ being left polaroid operators which satisfy property
$(\p)$, there exist (by Lemma \ref{lem1}) $A$-invariant closed
subspaces $M_1$ and $M_2$, and $B$-invariant closed subspaces
$N_1$ and $N_2$, such that

$$\X=M_1\oplus M_2, \Y=N_1\oplus N_2,
\hspace{2mm}\mbox{and}\hspace{2mm}
\X\overline{\otimes}\Y=M_1\overline{\otimes} N_1\oplus
M_1\overline{\otimes} N_2\oplus M_2\overline{\otimes} N_1 \oplus
M_2\overline{\otimes} N_2,$$ where the closed subspaces
$M_i\overline{\otimes} N_j$, $1\leq i,j \leq 2$, are $A\otimes
B$-invariant and, for some integers $d_1, d_2\geq 1$, $$(A-\mu
I)^{d_1}|_{M_1} =0 = (B-\nu I)^{d_2}|{N_1},
\hspace{2mm}\mbox{and}\hspace{2mm} (A-\mu I)|_{M_2}, (B-\nu I
)|_{N_2} \hspace{2mm}\mbox{are bounded below}.$$

\indent Let $d_1
+d_2=d$. Then, since $$A\otimes B- \lambda(I\otimes I)= (A-\mu I
)\otimes B + (\mu I \otimes (B-\nu I))= S + T
\hspace{2mm}\mbox{say},$$
$$\{A\otimes B-
\lambda(I\otimes I)\}^d= \sum_{k=0}^d \left(\begin{array}{crcl}d\\
 k\end{array}\right) S^k T^{d-k}$$ implies that $$ \{A\otimes B-
\lambda(I\otimes I)\}^d|_{M_i\overline{\otimes} N_j} =0; 1\leq i,j \leq
2\hspace{2mm}\mbox{and}\hspace{2mm} i,j\neq 2.$$ \noindent
Furthermore, since $\mu\notin\sigma_a(A|_{M_2})$ and
$\nu\notin\sigma_a(B|_{N_2})$,
$\lambda=\mu\nu\notin\sigma_a(A|_{M_2}\otimes
B|_{N_2})=\sigma_a(A\otimes B|_{M_2\overline{\otimes} N_2})$, and
hence $\{A\otimes B- \lambda(I\otimes I)\}|_{M_2\overline{\otimes}
N_2}$ is bounded below. Thus $\X\overline{\otimes} \Y$ is the
direct sum of two $A\otimes B$-invariant closed subspaces of
$\X\overline{\otimes} \Y$ such that the restriction of $A\otimes
B- \lambda(I\otimes I)$ to one of them is nilpotent and its
restriction to the other is bounded below. Apparently,
$asc(A\otimes B- \lambda(I\otimes I))\leq d<\infty$ and
$\{A\otimes B- \lambda(I\otimes I)\}^{d+1}(\X\overline{\otimes}
\Y)$ is closed; hence $A\otimes B$ is left polar at $\lambda$.\par

\indent $(ii).$ If $\lambda=0\in\iso\sigma_a(A\otimes B)$, then either
$(a)$ $0$ is not in one of $\sigma_a(A)$ and $\sigma_a(B)$, or
$(b)$ $0\in \sigma_a(A)\cap\sigma_a(B)$. If $(a)$ holds and
$0\notin\sigma_a(A)$, then $0\in\iso\sigma_a(B)$, $A$ is left
invertible and there exist $B$-invariant closed subspaces $N_1$
and $N_2$ such that $\Y=N_1\oplus N_2$, $B|_{N_1}$ is nilpotent
and $B|_{N_2}$ is bounded below. Since $\X\overline{\otimes} \Y= \X\overline{\otimes}
N_1 \oplus \X\overline{\otimes} N_2$, $A\otimes B|_{\X\overline{\otimes} N_1}$ is
nilpotent and $A\otimes B|_{X\overline{\otimes} N_2}$ is bounded below. Thus
$A\otimes B$ is left polar at $0$. Since a similar argument works
for the case in which $0\in\iso\sigma_a(A)$ and
$0\notin\sigma_a(B)$, we are left with case $(b)$. If
$0\in\sigma_a(A)\cap\sigma_a(B)$, then either $(b_1)$
$0\in\iso\sigma_a(A)\cap\iso\sigma_b(B)$, or $(b_2)$
$0\in\iso\sigma_a(A)\cap \rm{acc}\sigma_a(B)$, or $(b_3)$
$0\in\rm{acc}\sigma_a(A)\cap\iso\sigma_a(B)$. If $(b_1)$ holds,
then we copy the argument of $(i)$ above, with $\mu=\nu=0$, to
obtain $A\otimes B$ is left polar at $0$. If, instead, $(b_2)$
(respectively, $(b_3)$) holds, then $\sigma_a(A)=\{0\}$ (respectively,
$\sigma_a(B)=\{0\}$), and $A$ (respectively, $B$) is nilpotent. This
implies that $A\otimes B$ is nilpotent, hence left
polaroid.\end{demo}

\indent  Evidently, Theorem \ref{thm1} has a right polar
analogue. Observe that if an operator $T\in\b$ is polaroid (i.e.,
it is both left and right polaroid), then
$\iso\sigma(T)\cap\{\sigma_{UBW}(T)\cup\sigma_{LBW}(T)\}=\iso\sigma(T)\cap\sigma_{BW}(T)=\emptyset$.
In such a case, there exists an integer $d(\lambda)>0$ such that
$\rm{co-dim}((T-\lambda)\X + (T-\lambda)^{-d(\lambda)}(0))$ and
$\dim(((T-\lambda)^{-1}(0) \cap (T-\lambda)^{d(\lambda)}\X))$ are
both finite at every $\lambda\in\iso\sigma(T)$. Hence
 there exist $T$-invariant closed subspaces
$E_1$ and $E_2$ such that $\X=E_1\oplus E_2$,
$(T-\lambda)|_{E_1}$ is $d(\lambda)$-nilpotent and
$(T-\lambda)|_{E_2}$ is invertible at every
$\lambda\in\iso\sigma(T)$ ({\em cf.} \cite[Theorem 7]{Mu}). The
argument of the proof of Theorem \ref{thm1} implies the following.
\begin{cor}\label{cor1} \cite[Theorem 3]{DHK} $A$ and
$B$ polaroid implies $A\otimes B$ polaroid.\end{cor}

\indent The Hilbert space version of Theorem \ref{thm1} has a
$\tau_{AB}$ analogue.\par

\begin{thm}\label{thm3} If $A\in B(\H)$ and $B^*\in B(\K)$
are left polaroid Hilbert space operators, then $\tau_{AB}$ is
left polaroid\end{thm}\begin{demo}  To prove the  Theorm, one
argues as in the proof of \cite[Corollary 4]{DHK}: $B(B(\K),
B(\H))$ is an ultraprime Banach $(B(\K), B(\H))$ bimodule, and
hence $\tau_{AB}$ is just $A\otimes B^*$. Here the ultraprime
condition $||L_AR_B||=||A||||B||$ ensures that the operator norm
of the bimodule induces a uniform cross-norm on
$\H\overline{\otimes}\K$.\end{demo}

\markright{ \hskip5truecm \rm ENRICO BOASSO and B. P. DUGGAL}
\section {\sfstp Finitely left polaroid operators}

 Recall that upper semi-Fredholm operators have a Kato
decomposition: indeed, if $\lambda\not\in\sigma_{ab}(T)$, then
there exist $T$-invariant closed subspaces $E_1$ and $E_2$ such
that $H_0(T-\lambda)=H_0((T-\lambda)|_{E_1})=(T-\lambda)^{-d}(0)$
for some integer $d=\asc(T-\lambda)>0$, $\dim H_0(T-\lambda)= \dim
E_1 <\infty$, and $(T-\lambda)|_{E_2}$ is bounded below.
Apparently, if the operators $A$  and $B$  are finitely left
polaroid, then $A\otimes B$ is left polaroid. The finitely left
polaroid property does not transfer from $A, B$ to $A\otimes B$;
the problem, as one would expect, lies with the point $0$.\par

\begin{thm}\label{thm2}If $A$ and $B$ are finitely
left polaroid, then $A\otimes B$ is finitely left polaroid if
and only if $0\notin\iso\sigma_a(A\otimes
B)$.\end{thm}

\begin{demo} Recall, \cite[Lemma 5]{DDK}, that
$$\sigma_{ab}(A\otimes B)=\sigma_{ab}(A)\sigma_a(B) \cup
\sigma_a(A)\sigma_{ab}(B).$$ \noindent Suppose that
$0\neq\lambda\notin\iso\sigma_a(A\otimes B)$. Then there exist
non-zero $\mu\in\iso\sigma_a(A)$ and $\nu\in\iso\sigma_a(B)$ such
that $\mu\nu=\lambda$. If $A$ and $B$ are finitely left polaroid,
then $\mu\notin\sigma_{ab}(A)$ and $\nu\notin\sigma_{ab}(B)$.
Hence
$$\lambda\notin\sigma_{ab}(A\otimes B)\Longleftrightarrow
\lambda\in\Pi^l_0(A\otimes B).$$

\indent Now let $\lambda=0$. Since the
finitely left polaroid hypothesis on $A$ (respectively, $B$) implies that
$\X$ (respectively, $\Y$) has a direct sum decomposition of the type
considered in the proof of Theorem \ref{thm1} whenever
$0\in\iso\sigma_a(A)$ (respectively, $0\in\iso\sigma_a(B)$), it follows
from the argument of the proof of Theorem \ref{thm1}, see $(ii)$
of the proof, that $A\otimes B$ is left polaroid at $0$, with
$\alpha(A\otimes B)=\infty$. Hence $A\otimes B$ is not finitely
left polaroid at $0$.\end{demo}

\begin{rema}\label{rema11}{\em The upper semi-Fredholm spectrum
$\sigma_{SF_+}(T)$ of an operator $T$ satisfies the inclusion
$\sigma_{SF_+}(T)\subseteq \sigma_{ab}(T)$. Since
$0\notin\sigma_a(T)\setminus\sigma_{SF_+}(T)$ for every operator
$T$ (\cite[Lemma 4]{DDK}), $0\notin\sigma_a(A\otimes
B)\setminus\sigma_{ab}(A\otimes B)=\Pi_0^{\ell}(A\otimes B)$: this
provides an alternative proof of a part of Theorem
\ref{thm2}.}\end{rema}

\begin{rema}\label{rema111}\rm If the hypotheses of Theorem \ref{thm2} are satisfied,
then to prove Theorem  \ref{thm2} it is enough to consider the
case $0\in \sigma_a(A\otimes B)$; see Lemma \ref{lemma1}(i) and
Remark \ref{remark1}(iii). Observe that if $0\notin\iso
\sigma_a(A\otimes B)$, then $A\otimes B$ is finitely left
polaroid (by  Lemma \ref{lemma1}(i) and Remark
\ref{remark1}(iii)). If, instead, $A\otimes B$ is finitely left
polaroid and $0\in\iso \sigma_a(A\otimes B) \subseteq \iso
\sigma_a(A)\iso \sigma_a (B)=\Pi_0^l(A)\Pi_0^l(B)$, then
$0\in\Pi^{\ell}_0(A)$ or $0\in\Pi^{\ell}_0(B)$. However, if
$0\in\Pi_0^l(A)$ (respectively, $0\in\Pi^{\ell}_0(B)$), then,
since acc $\sigma_a(B)=\sigma_{ab}(B)\neq\emptyset$
(respectively, $\rm{acc }$ $\sigma_a(A)=\sigma_{ab}(A)\neq\emptyset$)
 and $\sigma_a(A\otimes B)=\sigma_a(A)\sigma_a(B)$, $0\in$ acc
$\sigma_a(A\otimes B)$, which is a contradiction. This provides
yet another  proof of Theorem \ref{thm2}.
\end{rema}

\indent The following remark is a supplement to the conclusions
of Theorem \ref{thm2}. In fact, given $A$ and $B$ two finitely
left polaroid operators, $\sigma_a( A\otimes B)$ will be fully
described in terms of the  Browder essential approximate point
spectrum and the set of finite left poles of the operators $A$
and $B$. Observe that  Remark \ref{remark101} describes
$\sigma_{ab}(A\otimes B)$  for finitely left polaroid operators
$A$ and $B$ .
\par
\markright{\hskip4.2truecm \rm Tensor product of  left polaroid operators }

\begin{rema}\label{rema12} \rm Let, as in Theorem \ref{thm2},  $A$ and $B$ be two finitely  left polaroid operators.\par
\noindent (i) If $0\notin \sigma_a(A)\cdot
\sigma_a(B)=\sigma_a(A\otimes B)$, then according to Theorem
\ref{thm2} and Lemma \ref{lemma1}(ii), acc $\sigma_a (A\otimes
B)=\sigma_{ab} (A\otimes B)$ and $\Pi_0^l(A\otimes B)=
\hbox{iso}\hskip,1truecm\sigma_a(A\otimes B)=
\Pi_0^l(A)\cdot\Pi_0^l(B)$. Same conclusions can be derived when
$0\in \hbox{acc}\hskip.1truecm \sigma_a(A)\setminus
\hbox{iso}\hskip.1truecm \sigma_a(B)$ or $0\in
\hbox{acc}\hskip.1truecm \sigma_a(B)\setminus
\hbox{iso}\hskip.1truecm \sigma_a(A)$.

\

\noindent (ii) If $0\in  \hbox{acc}\hskip.1truecm \sigma_a(A)\cap
\hbox{ iso}\hskip.1truecm \sigma_a(B)$, then according to the
last observation in Remark \ref{remark1} and Lemma
\ref{lemma1}(i),  acc $\sigma (A\otimes B)=\sigma_{ab}(A\otimes
B)$. In addition, according to Lemma \ref{lemma1}(ii),  $\iso
\sigma_a (A\otimes B)=\Pi_0^l(A\otimes
B)=\Pi_0^l(A)\cdot(\Pi_0^l(B)\setminus \{0 \})$. Similarly, if
$0\in  \hbox{acc}\hskip.1truecm \sigma_a(B)\cap \hbox{
iso}\hskip.1truecm \sigma_a(A)$, then $\iso \sigma_a (A\otimes
B)=\Pi_0^l(A\otimes B)=(\Pi_0^l(A)\setminus \{0
\})\cdot\Pi_0^l(B)$).

\

\noindent (iii) If  $0\in  \hbox{iso}\hskip.1truecm \sigma_a(A)$  and $0\notin\sigma_a(B)$, then
since $\sigma_a(A)\cdot \sigma_a(B)=\sigma_a(A\otimes B)$, a standard argument on convergent subsequences
proves that $0\in \iso \sigma_a (A\otimes B)$. Consequently, according to Lemma \ref{lemma1}(i)-(ii) and \cite[Lemma 5]{DDK},
 $\sigma_ {ab}(A\otimes B)= \hbox{acc}\hskip.1truecm \sigma_a(A\otimes B)\cup \{0\}$,
$I_0^a(A\otimes B)= \{0\}$, $\Pi_0^l(A\otimes B)=(\Pi_0^l(A)\setminus \{0\})\cdot \Pi_0^l(B)$,
$ \hbox{iso}\hskip.1truecm \sigma_a(A\otimes B)=\Pi_0^l(A)\cdot
\Pi_0^l(B)$ and $\hbox{acc}\hskip.1truecm \sigma_a(A\otimes B)=
 \sigma_{ab}(A)\cdot \sigma_{ab}(B)\cup  \sigma_{ab}(A)\cdot \Pi_0^l(B) \cup  (\Pi_0^l(A)\setminus \{0\})\cdot
 \sigma_{ab}(B)$.

 \

\noindent (iv) If $0\in  \hbox{iso}\hskip.1truecm \sigma_a(A)\cap  \hbox{iso}\hskip.1truecm \sigma_a(B)$, then
an argument similar to the one in (iii) proves that $ \hbox{iso}\hskip.1truecm \sigma_a(A\otimes B)=\Pi_0^l(A)\cdot
\Pi_0^l(B)$, $\Pi_0^l(A\otimes B)=(\Pi_0^l(A)\setminus \{0\})\cdot (\Pi_0^l(B)\setminus \{0\})$, $I_0^a(A\otimes B)= \{0\}$,
$\sigma_{ab}(A\otimes B)= \hbox{acc}\hskip.1truecm \sigma_a(A\otimes B)\cup \{0\}$ and
acc $\sigma_a(A\otimes B)=
 \sigma_{ab}(A)\cdot \sigma_{ab}(B)\cup
 \sigma_{ab}(A)\cdot (\Pi_0^l(B)\setminus \{0\})\cup
(\Pi_0^l(A)\setminus \{0\})\cdot
 \sigma_{ab}(B)$.\par
 \indent Note that the transfer property for finitely left polaroid operators
holds in (i) and (ii).
\end{rema}

\indent We consider next the elementary operator $\tau_{AB}$
(where, as before, $A\in\b$ and $B\in\B$).\par

\begin{thm}\label{thm31}  If $A$ and $B^*$ are
finitely left polaroid operators, then $\tau_{AB}$ is finitely
left polaroid if and only if
$0\notin\iso\sigma_a(\tau_{AB})$.\end{thm}\begin{demo}  Recall
that $\sigma_a(\tau_{AB})=\sigma_a(A)\sigma_a(B^*)$ and
$\sigma_{ab}(\tau_{AB})=\sigma_{ab}(A)\sigma_a(B^*)\cup\sigma_a(A)\sigma_{ab}(B^*)$
\cite[Proposition 4.1]{BDJ}. Now argue as in the proof of Theorem
\ref{thm2} to prove that $\tau_{AB}$ is finitely left polaroid at
every non-zero $\lambda\in\iso\sigma_a(\tau_{AB})$, and as in
Remark \ref{rema11} to prove that $\tau_{AB}$ is not finitely left
polaroid at $0\in\iso\sigma_a(\tau_{AB})$.\end{demo}

\indent Apparently,  an  alternative proof of Theorem \ref{thm31}
is obtained from an argument similar to the one in Remark
\ref{rema111}. Furthermore,  arguing  just as for the operator
$A\otimes B$ in Remark \ref{rema12}, it is possible to obtain a
complete characterization of the sets
 $\sigma_a (\ST)$, acc $\sigma_a
(\ST)$, $\iso \sigma_a (\ST)$, $I_0^a(\ST)$ and $\Pi_a^l(\ST)$,
 in terms of the corresponding sets for
$A$ and $B^*$. The details are left to the reader.\par

\medskip

\indent We end this section by studying perturbations of finitely
left polaroid operators by quasi-nilpotents.\par

\indent If $Q_1\in\b$ and $Q_2\in\B$ are quasi-nilpotents which
commute with $A\in\b$ and $B\in\B$ respectively, then
$(A+Q_1)\otimes (B+Q_2)= (A\otimes B)+Q$, where $Q=A\otimes Q_2 +
Q_1\otimes B + Q_1\otimes Q_2$ is a quasi-nilpotent which commutes
with $A\otimes B$. Since
$$\sigma_a((A+Q_1)\otimes
(B+Q_2))=\sigma_a(A+Q_1)\sigma_a(B+Q_2)=\sigma_a(A)\sigma_a(B)
\hspace{2mm}\mbox{and}$$
\begin{eqnarray*} &  & \sigma_{ab}((A+Q_1)\otimes
(B+Q_2))=\sigma_{ab}(A+Q_1)\sigma_a(B+Q_2)\cup
\sigma_a(A+Q_1)\sigma_{ab}(B+Q_2)\\ & = &
\sigma_{ab}(A)\sigma_a(B)\cup\sigma_a(A)\sigma_{ab}(B),\end{eqnarray*}
$A$ and $B$ finitely left polaroid implies $(A+Q_1)\otimes
(B+Q_2)$ finitely left polaroid at every
$0\neq\lambda\in\iso\sigma_a((A+Q_1)\otimes (B+Q_2))$.
Furthermore, since $A\otimes B=(A+Q_1)\otimes (B+Q_2)-Q$, $Q$ as
above, we have: \begin{cor}\label{cor2} If $A\in\b$ and $B\in\B$
are finitely left polaroid, and  $Q_1\in\b$ and $Q_2\in\B$ are
quasi-nilpotents which commute with $A$ and $B$ respectively, then
$(A+Q_1)\otimes (B+Q_2)$ is finitely left polaroid if and only if
$0\notin\iso\sigma_a((A+Q_1)\otimes (B+Q_2))$.\end{cor}

\markright{ \hskip5truecm \rm ENRICO BOASSO and B. P. DUGGAL}
\section {\sfstp An application}\setcounter{df}{0}

 For an operator
$T\in\b$, let $E^a(T)=\{\lambda\in\iso\sigma_a(T):
0<\alpha(T-\lambda)\}$ and $E^a_0(T)=\{\lambda\in E^a(T):
\alpha(T-\lambda)<\infty\}$. Recall that  $T$ is said to satisfy
$a$-Browder's theorem, $a$-Bt for short (respectively, generalized
$a$-Browder's theorem, $a$-gBt for short) if
$\sigma_a(T)\setminus\sigma_{aw}(T)=\Pi_0^l(T)$ (respectively,
$\sigma_a(T)\setminus\sigma_{UBW}(T)=\Pi^{\ell}(T)$). The
following equivalence is well known \cite[Theorem 2.2]{AZ}: {\em
$T$ satisfies $a$-Bt if and only if $T$ satisfies $a$-gBt}. $T$
satisfies $a$-Weyl's theorem, $a$-Wt for short (respectively,
generalized $a$-Weyl's theorem, $a$-gWt for short) if
$\sigma_a(T)\setminus\sigma_{aw}(T)=E_0^a(T)$ (respectively,
$\sigma_a(T)\setminus\sigma_{UBW}(T)=E^a(T)$). The following one
way implication holds: {\em $T$ satisfies $a$-gWt implies $T$
satisfies $a$-Wt}. Next generalized $a$-Weyl's theorem for $A\otimes B$
will be studied under the assumption $A$ and $B$ (finitely) left polaroid.\par
\markright{\hskip4.2truecm \rm Tensor product of  left polaroid operators }
\begin{thm}\label{thm4} Suppose that $A\in\b$
and $B\in\B$ satisfy $a$-Bt. If (i) $A, B$ are finitely left
polaroid, or (ii) $\X, \Y$ are Hilbert spaces and $A, B$ are left
polaroid, then $A\otimes B$ satisfies $a$-gWt if and only if
$\sigma_{aw}(A\otimes
B)=\sigma_{aw}(A)\sigma_a(B)\cup\sigma_a(A)\sigma_{aw}(B)$.\end{thm}

\begin{demo} If $A$ and $B$ satisfy $a$-Bt, then $A\otimes B$
satisfies $a$-Bt (consequently, also $a$-gBt) if and only if
$\sigma_{aw}(A\otimes
B)=\sigma_{aw}(A)\sigma_a(B)\cup\sigma_a(A)\sigma_{aw}(B)$
(\cite[Theorem 1]{DDK}). Thus $\sigma_a(A\otimes
B)\setminus\sigma_{UBW}(A\otimes B)=\Pi^{\ell}(A\otimes
B)\subseteq E^a(A\otimes B)$. Since either of the hypotheses
$(i)$ and $(ii)$ of the statement of the theorem implies
$A\otimes B$ is left polaroid, $E^a(A\otimes B)\subseteq
\Pi^{\ell}(A\otimes B)$. Hence $A\otimes B$ satisfies $a$-gWt.
The necessity being obvious from the implications $A\otimes B$
satisfies $a$-gWt implies $A\otimes B$ satisfies $a$-gBt implies
$A\otimes B$ satisfies $a$-Bt, the proof is complete.\end{demo}

\indent The finite left polaroid requirement in Theorem \ref{thm4} may be
relaxed in the case in which $A^*$ has SVEP on $\Pi^{\ell}(A)$
and $B^*$ has SVEP on
$\Pi^{\ell}(B)$.\par

\begin{thm}\label{thm51}Suppose that $A\in\b$ and
$B\in\B$ satisfy $a$-Bt. If $A^*$ has SVEP at points not in
$\sigma_{UBW}(A)$, $B^*$ has SVEP  at points not in
$\sigma_{UBW}(B)$ and $A, B$ are left polaroid, then $A\otimes B$
satisfies $a$-gWt if and only if $\sigma_{aw}(A\otimes
B)=\sigma_{aw}(A)\sigma_a(B)\cup\sigma_a(A)\sigma_{aw}(B)$.\end{thm}

\begin{demo}
The necessity follows as in the proof of Theorem \ref{thm4};
also, if $A$, $B$ satisfy $a$-Bt, and $\sigma_{aw}(A\otimes
B)=\sigma_{aw}(A)\sigma_a(B)\cup\sigma_a(A)\sigma_{aw}(B)$, then
$A\otimes B$ satisfies $a$-gBt. If $\mu\in\Pi^{\ell}(A)$, then
the SVEP hypothesis on $A^*$ implies implies the existence of
$A$-invariant subspaces $M_1$ and $M_2$ such that $\X=M_1\oplus
M_2$, $(A-\mu)|_{M_1}$ is nilpotent and $(A-\mu)|_{M_2}$ is
invertible (see the argument preceding Corllary \ref{cor1});
similarly, if $\nu\in\Pi^{\ell}(B)$, then the SVEP hypothesis on
$B^*$ implies the existence of $B$-invariant subspaces
$N_1$ and $N_2$ such that $\Y=N_1\oplus N_2$, $(B-\nu)|_{N_1}$ is
nilpotent and $(B-\nu)|_{N_2}$ is invertible. The slight changes
in the argument in the case in which one of $\mu$ and $\nu$ is
$0$ and the other is not a left pole being obvious, it follows
from the argument of the proof of Theorem \ref{thm1} that the
left polaroid property transfers from $A$ and $B$ to $A\otimes
B$. Hence, see the proof of Theorem \ref{thm4}, $A\otimes B$
satisfies $a$-gWt.\end{demo}

\indent Evidently, the operator $A\otimes B$
of Theorem \ref{thm51} satisfies (generalized Weyl's theorem,
$\sigma(A\otimes B)\setminus\sigma_{BW}(A\otimes B)= E(A\otimes
B)=\{\lambda\in\iso\sigma(A\otimes B): \lambda$ is an eigenvalue
of $A\otimes B \}$ and) $a$-Wt. More is true: $\sigma_a(A\otimes
B)\setminus\sigma_{UBW}(A\otimes B)=E(A\otimes B)$. To see this,
we observe that if the hypotheses of Theorem \ref{thm51} are
satisfied, then $\Pi^{\ell}(A)=\Pi(A)=E(A)$,
$\Pi^{\ell}(B)=\Pi(B)=E(B)$ and $\sigma_a(A\otimes
B)\setminus\sigma_{UBW}(A\otimes B)=\Pi^{\ell}(A\otimes B)=
E^a(A\otimes B)$. Evidently, $E(A\otimes B)\subseteq E^a(A\otimes
B)$. Let $\lambda\in E^a(A\otimes B)$. If $\lambda\neq 0$, then
there exists $\mu\in\iso\sigma_a(A)$ and $\nu\in\sigma_a(B)$ such
that $\mu\in\Pi^{\ell}(A)=\Pi(A)\subseteq E(A)$ and
$\nu\in\Pi^{\ell}(B)=\Pi(B)\subseteq E(B)$; hence $\lambda\in
E(A\otimes B)$. If, instead, $\lambda=0$, then either
$0\in\Pi^{\ell}(A)\cap \Pi^{\ell}(B)=\Pi(A)\cap \Pi(B)$, or $0$
is in one of $\Pi(A)$, $\Pi(B)$ and not in the other; in either
case $0\in E(A\otimes B)$. Hence $E(A\otimes B)\subseteq
E^a(A\otimes B)$. \par

\indent Theorem \ref{thm4} has a
$\tau_{AB}$ analogue.\par

\begin{cor}\label{cor3} Suppose that $A\in\b$
and $B^*\in B({\Y}^*)$ satisfy $a$-Bt. If (i) $A, B^*$ are
finitely left polaroid, or (ii) $\X, \Y$ are Hilbert spaces and
$A, B^*$ are left polaroid, then $\tau_{AB}$ satisfies $a$-gWt if
and only if
$\sigma_{aw}(\tau_{AB})=\sigma_{aw}(A)\sigma_a(B^*)\cup\sigma_a(A)\sigma_{aw}(B^*)$.\end{cor}\begin{demo}
To prove the corollary one argues as in the theorem above, using
Theorem \ref{thm3} and the fact that if $A$ and $B^*$ satisfy
$a$-Bt then $\tau_{AB}$ satisfies $a$-gBt if and only if
$\sigma_{aw}(\tau_{AB})=\sigma_{aw}(A)\sigma_a(B^*)\cup\sigma_a(A)\sigma_{aw}(B^*)$
(\cite[Theorem 4.5]{BDJ}).\end{demo}

\markright{ \hskip5truecm \rm ENRICO BOASSO and B. P. DUGGAL}
\baselineskip=12pt

\bigskip

\noindent \normalsize \rm Enrico Boasso\par
  \noindent  E-mail: enrico\_odisseo@yahoo.it \par
\medskip
\noindent B. P. Duggal\par
\noindent E-mail:  bpduggal@yahoo.co.uk
\end{document}